\documentclass{amsart}
\usepackage{amssymb}
\usepackage{amsmath}
\usepackage{amsfonts}

\newtheorem{theorem}{Theorem}
\theoremstyle{plain}

\newtheorem{lemma}{Lemma}

\numberwithin{equation}{section}
\input{tcilatex}

\begin{document}
\title[Uniform Parametrization ]{Uniform Parametrization in Pseudo-Complex
Hyperbolic $%
\mathbb{C}
^{n}$ Space}
\author{Minh Q. Truong}
\address{{\small Department of Biological and Physical Sciences }\\
St. Louis, MO 63105}
\email{mtruong@fontbonne.edu}
\subjclass{}
\keywords{Uniformization, Tangent Bundle, Pseudo-Hyperbolic Space}

\begin{abstract}
The parametrization theorem is derived in a flat $nD$ pseudo-complex affine
space. The pseudo-complex hyperbolic space accomodates $n$-number of
uncompactified time-like extra dimensions with sugnature $\left( s,r\right) $%
, where $s$ and $r$ are the numbers of minus and plus signs associated with
the diagonalized metric matrix. The main result of the theorem suggests a
uniform parametrization for both time-like and space-like dimensions. The
uniformization requirement preserves complex-hyperbolic inner product
associated with the space. As application, the elements of the space is
shown to be invariant under linear transformation.
\end{abstract}

\maketitle

\section{Introduction}

The standard approach in dealing with higher dimensional theories is to
consider almost exclusively space-like extra dimensions\cite{Antoniadis}.
Large extra-dimensions have been used to address the hierarchy problem,
whereas Higgs mass is proven to be finite\cite{Higgsmass}.The effects of the
extra space-like dimensions have been examined in the context of 4D
superspace formalism\cite{Truong}. However, there is no priory reason why
extra time-like dimensions cannot exist. Time-like extra dimensions have
been ignored due to serious conflicts with causality and unitarity\cite%
{Yndurain}\cite{Senjanovic}\cite{Erdem}. Time-like extra dimensions have
been used within the framework of brane world models, as an alternative in
reconciling the mass hierarchy problem\cite{Chaichian}. It has been shown
that extra time-like dimensional theories can yield tachyons-free modes\cite%
{Quiros1}\cite{Quiros2}. The main result of this paper is a generalization
of some of the results obtained in\cite{Truong2}, specifically, the
constraint of the time-like dimensions. The complex nature of the
pseudo-complex hyperbolic affine $%
\mathbb{C}
^{n}$ space and metric function lead to the uniform parametrization theorem.

In the next section, the hyperbolic parametrization of elements, $%
p^{A}\left( \psi \right) \in M\subset $ $%
\mathbb{C}
^{n}$ space and $\overset{\cdot }{p}_{A}\left( \psi \right) \in W\subset
T_{p^{A}\left( \psi \right) }\left( M\right) $, where $M$ and $W$ are some
open subsets of $%
\mathbb{R}
^{s,r}$ space and its associated tangent space $T_{p^{A}\left( \psi \right)
}\left( M\right) $, are derived and contained in the lemma. The results from
the lemma are used in one of the inductive steps associated with the
uniformization theorem.

\pagebreak

\subsection{Mathematics}

This section serves as a review on higher dimensional manifold\cite{Oneille}%
. Consider a smooth $n$-dimensional differential manifold $M$ $\ $equipped
with an atlas of \ charts $\left( U_{A},p_{A}\right) $ and%
\begin{equation}
p_{A}:U_{A}\rightarrow 
\mathbb{R}
^{n},
\end{equation}%
where $U_{A}\subset M$ is an open contractible subset. Equivalently, a point 
$p^{A}$ in $M$ is be parametrized by $\psi \in 
\mathbb{R}
$%
\begin{equation}
p^{A}:%
\mathbb{R}
\rightarrow M.
\end{equation}%
Since $U_{A}$ is an open contractible subset of $M,$ then there is a
restriction on the diffeomorphism from $T(U_{A})$ to $U_{A}\times 
\mathbb{R}
^{n}$ to just linear isomorphism from $T_{p^{A}\left( \psi \right) }\left(
U_{A}\right) $ to $\left\{ p^{A}\left( \psi \right) \right\} \times 
\mathbb{R}
^{n}.$ The tangent bundle $T\left( M\right) $ of $M$ defined as the disjoint
union of the tangent spaces of $M$ 
\begin{equation}
T\left( M\right) =\coprod\limits_{p^{A}\left( \psi \right) \in
M}T_{p^{A}\left( \psi \right) }\left( M\right) =\bigcup\limits_{p^{A}\left(
\psi \right) \in M}\left\{ p^{A}\left( \psi \right) \right\} \times
T_{p^{A}\left( \psi \right) }\left( M\right) ,
\end{equation}%
where $T_{p^{A}\left( \psi \right) }\left( M\right) $ denotes the tangent
space to $M$ at the point $p^{A}\left( \psi \right) .$ By inspection, the
dimensionality of the tangent bundle $T\left( M\right) $ is twice the
dimension of the differential manifold $M.$ The primary purpose of the
tangent bundle $T\left( M\right) $ is to provide a domain and range for the
derivative of a smooth function, i.e.%
\begin{equation}
f:M\longrightarrow W,
\end{equation}%
where $M$ and $W$ are some smooth differential manifolds. The derivative of
the function $f,$ $D$ $f$ is also a smooth function%
\begin{equation}
Df:T\left( M\right) \rightarrow T\left( W\right) .
\end{equation}%
Elements of $\ T\left( M\right) $ are pairs of the forms $\left( p^{A}\left(
\psi \right) ,\overset{\cdot }{p}^{A}\left( \psi \right) \right) ,$ where 
\begin{equation}
\overset{\cdot }{p}^{A}\left( \psi \right) =\frac{d\overset{\cdot }{p}%
^{A}\left( \psi \right) }{d\psi }.
\end{equation}%
The elements or points on a smooth differential manifold $M$ can be obtained
via a natural projection map $\pi $ defined by%
\begin{equation}
\pi :T\left( M\right) \rightarrow M,
\end{equation}%
i.e. 
\begin{equation}
\pi \left( p^{A}\left( \psi \right) ,\overset{\cdot }{p}^{B}\left( \psi
\right) \right) =p^{A}\left( \psi \right) .
\end{equation}%
The mapping 
\begin{equation}
\widetilde{p}_{\widetilde{A}}:\pi ^{-1}\left( U_{A}\right) \rightarrow 
\mathbb{R}
^{2n},
\end{equation}%
defined by 
\begin{equation}
\widetilde{p}_{\widetilde{A}}\left( p^{A}\left( \psi \right) ,v^{B}\partial
_{B}\right) =\left( p^{A}\left( \psi \right) ,v^{B}\right) ,
\end{equation}%
where $\widetilde{A}=1,...,2n$ and $A,B=1,...,n.$ The tangent bundle of $M$
is itself a manifold of dimensionality of $2n,$ and provides the domain and
range through the derivative map%
\begin{equation}
Df:T\left( M\right) \rightarrow T\left( W\right) .
\end{equation}%
The higher-order tangent bundle can be recursively defined by the relation 
\begin{equation}
T^{p}\left( M\right) =T\left( T^{p-1}\left( M\right) \right) ,
\label{recursive}
\end{equation}%
where $T^{p}$ is $p$-order tangent bundle, provides the domain and range
through the $p$-derivative map%
\begin{equation}
D^{p}f:T^{p}\left( M\right) \rightarrow T^{p}\left( W\right) .
\end{equation}%
The dimensionality of higher-order tangent bundle can be obtained via
equation $\left( \ref{recursive}\right) .$ The second-order tangent bundle%
\begin{equation}
T\left( T\left( M\right) \right) =\bigcup\limits_{p^{A}\left( \psi \right)
\in T\left( M\right) }\left\{ p^{A}\left( \psi \right) \right\} \times
T_{p^{A}\left( \psi \right) }\left( T\left( M\right) \right) ,
\end{equation}%
has dimensionality of $4n.$ It is straight forward to see that the $p$-order
tangent bundle 
\begin{equation}
T^{p}\left( M\right) =T\left( T^{p-1}\left( M\right) \right)
=\bigcup\limits_{p^{A}\left( \psi \right) \in T^{p-1}\left( M\right)
}\left\{ p^{A}\left( \psi \right) \right\} \times T_{p^{A}\left( \psi
\right) }\left( T^{p-1}\left( M\right) \right) ,
\end{equation}%
has dimensionality of $2^{p}\cdot n.$

\begin{lemma}
Show that the hyperbolic parametrized forms of $p_{A,l+1-s}\left( \psi
\right) $ and $\overset{\cdot }{p}_{A,l+1-s}\left( \psi \right) $ are given
by%
\begin{eqnarray}
p_{A,l+1-s}\left( \psi \right) &=&\sqrt{\frac{r}{s}}R_{eff}\tsum%
\nolimits_{i=1}^{s}\sinh \left( \sqrt{sr}\psi \right) \widehat{t}_{i} \\
&&+R_{eff}\tsum\nolimits_{j=s+1}^{k+1}\cosh \left( \sqrt{sr}\psi \right) 
\widehat{x}_{j},  \notag
\end{eqnarray}%
and%
\begin{eqnarray}
\overset{\cdot }{p}_{A,l+1-s}\left( \psi \right)
&=&rR_{eff}\tsum\nolimits_{i=1}^{s}\cosh \left( \sqrt{sr}\psi \right) 
\widehat{t}_{i} \\
&&+\sqrt{sr}R_{eff}\tsum\nolimits_{j=s+1}^{k+1}\sinh \left( \sqrt{sr}\psi
\right) \widehat{x}_{j},  \notag
\end{eqnarray}%
subjected to the following initial condition%
\begin{equation}
p_{A,n}\left( 0\right) =\left( t_{i}\left( 0\right) ,x_{j}\left( 0\right)
\right) =\left( \underset{s}{\underbrace{0,...,0}},\underset{k+1-s}{%
\underbrace{R_{eff},...,R_{eff}}}\right) ,  \label{initial condition}
\end{equation}%
where $i\in \left\{ 1,...,s\right\} $ and $j\in \left\{ s+1,...,k+1\right\}
. $
\end{lemma}

\begin{proof}
The $nD$ flat pseudo-complex affine hyperbolic space $H^{s,r}$ is a subset
of $%
\mathbb{C}
^{n},$ and is defined as%
\begin{equation}
H^{s,r}=\left\{ \left( t_{i},x_{j}\right)
:-\tsum\nolimits_{i=1}^{s}t_{i}^{2}+\tsum%
\nolimits_{j=s+1}^{k+1}x_{j}^{2}=R^{2}\right\} 
\end{equation}%
where $R_{eff}=f(r)R,$ $f(r)$ is a scaling function of $r,$ $R_{eff}$ is the
effective positive constant curvature of $H^{s,r}$ affine space. For the
case of one extra time-like dimension, the space defined in\cite{Truong2}can
be shown to be equivalent to a $5D$ $AdS$ space. The hyperbolic parametrized 
$p_{A,l+1-s}\left( \psi \right) $ and $\overset{\cdot }{p}_{A,l+1-s}\left(
\psi \right) $ can be obtained by solving the following systems of
differential equations%
\begin{eqnarray}
\overset{\cdot }{x}_{j} &=&\tsum\nolimits_{i=1}^{s}t_{i},\text{ \ }
\label{constraint 3} \\
\overset{\cdot }{t}_{i} &=&\tsum\nolimits_{j=s+1}^{k+1}x_{j}  \notag \\
\text{ \ \ \ \ }p\circledast \overset{\cdot }{p} &=&-\tsum%
\nolimits_{i=1}^{s}t_{i}\overset{\cdot }{t}_{i}+\tsum%
\nolimits_{j=s+1}^{k+1}x_{j}\overset{\cdot }{x}_{j}  \label{constraint 4} \\
p\circledast \overset{\cdot }{p} &=&-\tsum\nolimits_{j=s+1}^{k+1}\tsum%
\nolimits_{i=1}^{s}t_{i}x_{j}+\tsum\nolimits_{j=s+1}^{k+1}\tsum%
\nolimits_{i=1}^{s}x_{j}t_{i}=0  \notag
\end{eqnarray}%
By inspection, the systems of differential equations satisfy the following
constrained relations: 
\begin{eqnarray}
p_{A,l+1-s}\left( \psi \right)  &=&\left( t_{i},x_{j}\right) , \\
p_{A,l+1-s}\left( \psi \right) \circledast p_{A,l+1-s}\left( \psi \right) 
&=&-\tsum\nolimits_{i=1}^{s}t_{i}^{2}+\tsum%
\nolimits_{j=s+1}^{k+1}x_{j}^{2}=R^{2},  \notag \\
p_{A,l+1-s}\left( \psi \right) \circledast \overset{\cdot }{p}%
_{A,l+1-s}\left( \psi \right) 
&=&-\tsum\nolimits_{i=1}^{s}\tsum\nolimits_{i^{\prime }=1}^{s}t_{i}\overset{%
\cdot }{t}_{i^{\prime }}\delta _{j,j^{\prime }}  \notag \\
&&+\tsum\nolimits_{j=s+1}^{k+1}\tsum\nolimits_{j^{\prime }=s+1}^{k+1}x_{j}%
\overset{\cdot }{x}_{j^{\prime }}\delta _{j,j^{\prime }} \\
p_{A,l+1-s}\left( \psi \right) \circledast \overset{\cdot }{p}%
_{A,l+1-s}\left( \psi \right)  &=&-\tsum\nolimits_{i=1}^{s}t^{i}\overset{%
\cdot }{t}_{i}+\tsum\nolimits_{j=s+1}^{k+1}x^{j}\overset{\cdot }{x}_{j}=0, 
\notag
\end{eqnarray}%
where $i\in \left\{ 1,...,s\right\} $ and $j\in \left\{ s+1,...,k+1\right\} .
$ The vectors $p^{A}$ and $\overset{\cdot }{p}_{A}$are said to be $H^{s,r}$
perpendiculars w.r.t. each other. By inspection, the hyperbolic
parametrization for elements of $p^{A}$ and $\overset{\cdot }{p}_{A}$can be
obtained by solving the following systems of differential equations,%
\begin{eqnarray}
\overset{\cdot }{x}_{j}\left( \psi \right)  &=&t_{1}\left( \psi \right)
+\cdot \cdot \cdot +t_{s}\left( \psi \right) =\sum_{i=1}^{s}t_{i}  \notag \\
\overset{\cdot }{x}_{j}\left( \psi \right)  &=&\sum_{i=1}^{s}t_{i}\left(
\psi \right) ,  \label{constraint 5}
\end{eqnarray}%
and%
\begin{equation}
\overset{\cdot }{t}_{i}\left( \psi \right)
=\tsum\nolimits_{j=s+1}^{k+1}x_{j}\left( \psi \right) .  \label{constraint 6}
\end{equation}%
Taking the derivative of equation $\left( \ref{constraint 5}\right) $ and
using the constrained equation $\left( \ref{constraint 6}\right) ,$we have 
\begin{eqnarray}
\overset{\cdot \cdot }{x}_{j}\left( \psi \right)  &=&\tsum\nolimits_{i=1}^{s}%
\overset{\cdot }{t}_{i}\left( \psi \right)   \notag \\
\overset{\cdot \cdot }{x}_{j}\left( \psi \right) 
&=&\tsum\nolimits_{i=1}^{s}\left( \tsum\nolimits_{j=s+1}^{k+1}x_{j}\left(
\psi \right) \right)   \notag \\
\overset{\cdot \cdot }{x}_{j}\left( \psi \right)  &=&s\left(
\tsum\nolimits_{j=s+1}^{k+1}x_{j}\left( \psi \right) \right)   \notag \\
\overset{\cdot \cdot }{x}_{j}\left( \psi \right)  &=&srx_{j}\left( \psi
\right)   \notag \\
\overset{\cdot \cdot }{x}_{j}\left( \psi \right) -srx_{j}\left( \psi \right)
&=&0.  \label{2nd order constrained dfe1.1.2}
\end{eqnarray}%
The solution to equation $\left( \ref{2nd order constrained dfe1.1.2}\right) 
$ takes a general form of 
\begin{equation}
x_{j}\left( \psi \right) =A_{j}e^{\sqrt{sr}\psi }+B_{j}e^{-\sqrt{sr}\psi }%
\text{ \ \ \ \ \ }\forall j\in \left\{ s+1,...k+1\right\} ,  \label{Sol 3}
\end{equation}%
where $A_{j}$ and $B_{j}$ are arbitrary constants and will be determined by
the initial condition. To determine the constants, take the derivative of
equation $\left( \ref{Sol 3}\right) $, and substituting in for equation $%
\left( \ref{constraint 5}\right) $ 
\begin{equation}
\overset{\cdot }{x}_{j}\left( \psi \right) =\sqrt{sr}A_{j}e^{\sqrt{sr}\psi }-%
\sqrt{sr}B_{j}e^{-\sqrt{sr}\psi }=\tsum\nolimits_{i=1}^{s}t_{i}\left( \psi
\right) .  \label{1st order dfe 3}
\end{equation}%
Imposing the temporal initial condition $\left( \ref{initial condition}%
\right) $to equation $\left( \ref{1st order dfe 3}\right) ,$ yields%
\begin{eqnarray}
\overset{\cdot }{x}_{j}\left( 0\right)  &=&\sqrt{sr}A_{j}e^{\sqrt{sr}\cdot
0}-\sqrt{sr}B_{j}e^{-\sqrt{sr}\cdot 0}  \notag \\
&=&\tsum\nolimits_{i=1}^{s}t_{i}\left( 0\right) =0\text{ \ \ \ }  \notag \\
&\Rightarrow &\text{ }A_{j}-B_{j}=0\text{ \ }\Rightarrow \text{ }A_{j}=B_{j}%
\text{ \ \ \ \ \ }\forall j\in \left\{ s+1,...k+1\right\} \text{\ }.
\end{eqnarray}%
The spatial solution $\left( \ref{Sol 3}\right) $ becomes%
\begin{equation}
x_{j}\left( \psi \right) =A_{j}\left( e^{\sqrt{sr}\psi }+e^{-\sqrt{sr}\psi
}\right) .  \label{Sol 3.1}
\end{equation}%
Imposing the spatial initial condition on equation $\left( \ref{Sol 3.1}%
\right) $ 
\begin{eqnarray}
x_{j}\left( 0\right)  &=&A_{j}\left( e^{\sqrt{sr}\cdot 0}+e^{-\sqrt{sr}\cdot
0}\right) =R_{eff}  \notag \\
&\Rightarrow &\text{ }2A_{j}=R_{eff}\text{ }\Rightarrow \text{ }A_{j}=\frac{%
R_{eff}}{2}\text{ \ \ }\forall j\in \left\{ s+1,...k+1\right\} .
\end{eqnarray}%
The particular spatial solution yields 
\begin{eqnarray}
x_{j}\left( \psi \right)  &=&R_{eff}\left( \frac{e^{\sqrt{sr}\psi }+e^{-%
\sqrt{sr}\psi }}{2}\right)   \notag \\
x_{j}\left( \psi \right)  &=&R_{eff}\cosh \left( \sqrt{sr}\psi \right) .%
\text{ \ }  \label{spatial sol 1.1.1}
\end{eqnarray}%
Using the spatial solution $\left( \ref{spatial sol 1.1.1}\right) ,$ the $s$%
-number of temporal solutions contained in equation $\left( \ref{constraint
6}\right) $ can now be obtained from%
\begin{eqnarray}
\overset{\cdot }{t}_{i}\left( \psi \right)  &=&\left( k+1-s\right)
x_{j}\left( \psi \right)   \notag \\
\overset{\cdot }{t}_{i}\left( \psi \right)  &=&rx_{j}\left( \psi \right)
=x_{j}\left( \psi \right) =rR_{eff}\cosh \left( \sqrt{sr}\psi \right)  
\notag \\
\overset{\cdot }{t}_{i}\left( \psi \right)  &=&rR_{eff}\cosh \left( \sqrt{sr}%
\psi \right) ,
\end{eqnarray}%
where $r=\left( k+1-s\right) .$ For each $i=1,...,s$, we have the following
system of differential equations%
\begin{eqnarray}
\frac{dt_{1}\left( \psi \right) }{d\psi } &=&rR_{eff}\cosh \left( \sqrt{sr}%
\psi \right) \text{ \ }\Rightarrow \text{ \ }t_{1}\left( \psi \right) =\frac{%
rR_{eff}}{\sqrt{sr}}\sinh \left( \sqrt{sr}\psi \right) +C_{1}, \\
\frac{dt_{2}\left( \psi \right) }{d\psi } &=&rR_{eff}\cosh \left( \sqrt{sr}%
\psi \right) \text{ \ }\Rightarrow \text{ \ }t_{2}\left( \psi \right) =\frac{%
rR_{eff}}{\sqrt{sr}}\sinh \left( \sqrt{s}\psi \right) +C_{2}, \\
\frac{dt_{s}\left( \psi \right) }{d\psi } &=&rR_{eff}\cosh \left( \sqrt{sr}%
\psi \right) \text{ \ }\Rightarrow \text{ \ }t_{s}\left( \psi \right) =\frac{%
rR_{eff}}{\sqrt{sr}}\sinh \left( \sqrt{s}\psi \right) +C_{s}.
\end{eqnarray}%
Imposing the temporal initial condition on solutions on the above system of
differential equations$,$ yields vanishing constants $C_{i},$ $\forall i\in
\left\{ 1,...,s\right\} .$ The temporal solutions become%
\begin{equation}
t_{i}\left( \psi \right) =\sqrt{\frac{r}{s}}R_{eff}\sinh \left( \sqrt{sr}%
\psi \right) ,\text{ \ \ \ \ }\forall i\in \left\{ 1,...,s\right\} .
\label{Sol s.1}
\end{equation}%
Hyperbolically parameterized by $\psi ,$ the temporal and spatial components
of \ $p^{A}\in M\subset $ $H^{s,r}$ can finally be prescribed as%
\begin{eqnarray}
p_{A,l+1-s}\left( \psi \right)  &=&\sqrt{\frac{r}{s}}R_{eff}\tsum%
\nolimits_{i=1}^{s}\sinh \left( \sqrt{sr}\psi \right) \widehat{t}_{i}
\label{element of H space 1} \\
&&+R_{eff}\tsum\nolimits_{j=s+1}^{k+1}\cosh \left( \sqrt{sr}\psi \right) 
\widehat{x}_{j}.  \notag
\end{eqnarray}%
Similarly, taking the derivative of equation $\left( \ref{element of H space
1}\right) $ yields the perpendicular to $p^{A}$ 
\begin{eqnarray}
\overset{\cdot }{p}_{A,l+1-s}\left( \psi \right)  &=&\sqrt{\frac{r}{s}}\sqrt{%
sr}R_{eff}\tsum\nolimits_{i=1}^{s}\cosh \left( \sqrt{sr}\psi \right) 
\widehat{t}_{i}  \notag \\
&&+\sqrt{sr}R_{eff}\tsum\nolimits_{j=s+1}^{k+1}\sinh \left( \sqrt{sr}\psi
\right) \widehat{x}_{j},
\end{eqnarray}%
where $\widehat{t}_{i}$ and $\widehat{x}_{j}$ are unit vectors pointing in
the time-like and space-like directions, respectively.
\end{proof}

\begin{theorem}
Uniform Parametrization Theorem:

Consider $p_{A,n}\left( \psi \right) \in M\subset H^{s,r}\subset 
\mathbb{C}
^{n}$ be a smooth mapping, where $M$ is some open subset of $H^{s,r}$ and $%
\psi \in 
\mathbb{R}
$. The pseudo-complex hyperbolic affine space is defined as $%
H^{s,n-s}=\left\{ \left( t_{i},x_{j}\right)
:-\sum_{i=1}^{s}t_{i}^{2}+\sum_{j=s+1}^{n}x_{j}^{2}=R^{2}\right\} ,$ where $%
s $ and $r=n-s$ are the numbers of minuses and pluses dictated by the metric
function. The $p_{A,n}$ and $\overset{\cdot }{p}_{A,n}=\frac{dp_{A,n}}{d\psi 
}$are elements of $H^{s,n-s}$ space and its associated tangent space $%
T_{p_{A,n}}(M),$ respectively, along with the initial condition $\left( \ref%
{initial condition}\right) $. The elements of the respective spaces are
defined by $p_{A,n}\left( \psi \right) =\left( t_{i}\left( \psi \right)
,x_{j}\left( \psi \right) \right) $ and $\overset{\cdot }{p}_{A,n}\left(
\psi \right) =\left( \overset{\cdot }{t}_{i}\left( \psi \right) ,\overset{%
\cdot }{x}_{j}\left( \psi \right) \right) ,$ where $i\in \left\{
1,...,s\right\} $ and $j\in \left\{ s+1,...,n\right\} .$ Then elements of $%
H^{s,n-s}$ and its tangent space $T_{p_{A,n}}(M)$ must be uniformly
parametrized by parametrization parameter $\psi .$
\end{theorem}

\begin{proof}[Proof of the Main Theorem]
Proof is by induction. Let us redefine the index of the mapping by $n=k+1,$
the element of $H^{s,k+1-s}$ becomes $p_{A,k+1}\left( \psi \right) \in
M\subset H^{s,k+1-s}\subset 
\mathbb{C}
^{k+1}.$ For $k=1$, we have $p_{A,1+1-1}\left( \psi \right) =\left(
t_{i}\left( \psi \right) ,x_{j}\left( \psi \right) \right) $ and $\overset{%
\cdot }{p}_{A,2}\left( \psi \right) =\left( \overset{\cdot }{t}_{i}\left(
\psi \right) ,\overset{\cdot }{x}_{j}\left( \psi \right) \right) ,$ where $%
i=1$, $j=2$ and $A=1,2.$ With the given metric function, the invariant
squared of $p_{A,2}\left( \psi \right) $ yields 
\begin{equation}
p^{2}=p_{A,2}\left( \psi \right) \circledast p_{A,2}\left( \psi \right)
=-t_{1}^{2}\left( \psi \right) +x_{2}^{2}\left( \psi \right) =R^{2}.
\label{inner product 1}
\end{equation}%
The implicit differentiation of $p^{2\text{ }}$yields%
\begin{equation}
-t_{1}\overset{\cdot }{t}_{1}+x_{2}\overset{\cdot }{x}_{2}=0.
\label{Ortho 1}
\end{equation}%
Equation $\left( \ref{Ortho 1}\right) $ is equivalent to taking the $H^{1,1}$
inner product of $p_{A,2}$ and $\overset{\cdot }{p}_{A,2},$ thus $%
p_{A,2}\left( \psi \right) \circledast \overset{\cdot }{p}_{A,2}\left( \psi
\right) =-t_{1}\overset{\cdot }{t}_{1}+x_{2}\overset{\cdot }{x}_{2}=0.$
Equation $\left( \ref{Ortho 1}\right) $ is satisfied by using the following
differential equations, $\overset{\cdot }{x}_{2}=t_{1}$ and $\overset{\cdot }%
{t}_{1}=x_{2}.$ Solving these two dfe's for temporal and spatial components,
and imposing the initial condition, $p_{A,2}\left( 0\right) =\left(
t_{1}\left( 0\right) ,x_{2}\left( 0\right) \right) =\left( 0,R_{eff}\right) ,
$ yields $t_{1}=R_{eff}\sinh \psi $ and $x_{2}=R_{eff}\cosh \psi .$ Hence,
the elements $p_{A,2}$ and $\overset{\cdot }{p}_{A,2}$take the following
hyperbolic parametrized forms, $p_{A,2}\left( \psi \right) =\left(
R_{eff}\sinh \psi ,R_{eff}\cosh \psi \right) $ and $\overset{\cdot }{p}%
_{A,2}\left( \psi \right) =\left( R_{eff}\cosh \psi ,R_{eff}\sinh \psi
\right) .$ The elements of $H^{1,1}$ and its associated tangent space $%
T_{p_{A,2}}(M)$ can easily be shown to satisfy equations $\left( \ref{inner
product 1}\right) $ and $\left( \ref{Ortho 1}\right) .$ Hence, true for $k=1.
$

Assume true for $k=l,$ we have the following relations: 
\begin{eqnarray}
p_{A,l+1-s}\left( \psi \right) &=&\left( t_{i},x_{j}\right) , \\
p_{A,l+1-s}\left( \psi \right) \circledast p_{A,l+1-s}\left( \psi \right)
&=&-\tsum\nolimits_{i=1}^{s}t_{i}^{2}+\tsum%
\nolimits_{j=s+1}^{k+1}x_{j}^{2}=R^{2}, \\
p_{A,l+1-s}\left( \psi \right) \circledast \overset{\cdot }{p}%
_{A,l+1-s}\left( \psi \right)
&=&-\tsum\nolimits_{i=1}^{s}\tsum\nolimits_{i^{\prime }=1}^{s}t_{i}\overset{%
\cdot }{t}_{i^{\prime }}\delta _{j,j^{\prime }}  \notag \\
&&+\tsum\nolimits_{j=s+1}^{k+1}\tsum\nolimits_{j^{\prime }=s+1}^{k+1}x_{j}%
\overset{\cdot }{x}_{j^{\prime }}\delta _{j,j^{\prime }}  \notag \\
p_{A,l+1-s}\left( \psi \right) \circledast \overset{\cdot }{p}%
_{A,l+1-s}\left( \psi \right) &=&-\tsum\nolimits_{i=1}^{s}t_{i}\overset{%
\cdot }{t}_{i}+\tsum\nolimits_{j=s+1}^{k+1}x_{j}\overset{\cdot }{x}_{j}=0,
\end{eqnarray}%
where $i\in \left\{ 1,...,s\right\} $ and $j\in \left\{ s+1,...,k+1\right\}
. $ From the lemma, the assumed hyperbolic parametrized forms are given as%
\begin{eqnarray}
p_{A,l+1-s}\left( \psi \right) &=&\sqrt{\frac{r}{s}}R_{eff}\tsum%
\nolimits_{i=1}^{s}\sinh \left( \sqrt{sr}\psi \right) \widehat{t}_{i} \\
&&+R_{eff}\tsum\nolimits_{j=s+1}^{k+1}\cosh \left( \sqrt{sr}\psi \right) 
\widehat{x}_{j},  \notag
\end{eqnarray}%
and%
\begin{eqnarray}
\overset{\cdot }{p}_{A,l+1-s}\left( \psi \right)
&=&rR_{eff}\tsum\nolimits_{i=1}^{s}\cosh \left( \sqrt{sr}\psi \right) 
\widehat{t}_{i} \\
&&+\sqrt{sr}R_{eff}\tsum\nolimits_{j=s+1}^{k+1}\sinh \left( \sqrt{sr}\psi
\right) \widehat{x}_{j},  \notag
\end{eqnarray}%
where $R_{eff}=\frac{1}{\sqrt{r}}R.$ We could also assume the following
systems of differential equations 
\begin{eqnarray}
\overset{\cdot }{x}_{j} &=&\tsum\nolimits_{i=1}^{s}t_{i},\text{ \ and}
\label{1st order dfe 3.1} \\
\overset{\cdot }{t}_{i} &=&\tsum\nolimits_{j=s+1}^{k+1}x_{j}\text{ \ \ }
\label{1st order dfe 3.1.2} \\
&\Longrightarrow &\text{ \ \ }p_{A,l+1-s}\left( \psi \right) \circledast 
\overset{\cdot }{p}_{A,l+1-s}\left( \psi \right)
=-\tsum\nolimits_{i=1}^{s}t_{i}\overset{\cdot }{t}_{i}+\tsum%
\nolimits_{j=s+1}^{k+1}x_{j}\overset{\cdot }{x}_{j}  \notag \\
&=&-\tsum\nolimits_{j=s+1}^{k+1}\tsum\nolimits_{i=1}^{s}t_{i}x_{j}+\tsum%
\nolimits_{j=s+1}^{k+1}\tsum\nolimits_{i=1}^{s}x_{j}t_{i}=0  \notag
\end{eqnarray}

We need to show true for $k=l+1,$%
\begin{equation*}
p_{A,l+2-s}\left( \psi \right) =\left( t_{i},x_{j}\right) ,\text{ \ \ \ \ \ }%
i\in \left\{ 1,...,s,s+1\right\} \text{ and }j\in \left\{
s+2,...,k+1,k+2\right\} .
\end{equation*}%
The $H^{s+1,l+2-s}$ inner product of $p_{A,l+2-s}\left( \psi \right) ,$
yields%
\begin{eqnarray}
p_{A,l+2-s}\left( \psi \right) \circledast p_{A,l+2-s}\left( \psi \right)
&=&-\tsum\nolimits_{i=1}^{s+1}t_{i}^{2}+\tsum%
\nolimits_{j=s+2}^{k+2}x_{j}^{2}=R^{2}  \label{x} \\
&=&-\tsum\nolimits_{i=1}^{s}t_{i}^{2}+\tsum%
\nolimits_{j=s+2}^{k+2}x_{j}^{2}-t_{s+1}^{2}.  \label{l+2 inner product}
\end{eqnarray}%
We shift the index $j^{\prime }=j-1,$%
\begin{equation}
p_{A,l+2-s}\left( \psi \right) \circledast p_{A,l+2-s}\left( \psi \right) =%
\underset{R^{2}}{\underbrace{-\tsum\nolimits_{i=1}^{s}t_{i}^{2}+\tsum%
\nolimits_{j-1=s+1}^{k+1}x_{j}^{2}}}-t_{s+1}^{2}.
\label{l+2 inner product 1}
\end{equation}%
To show $t_{s+1}^{2}=0,$ we add $t_{s+1}$ to both sides of equation $\left( %
\ref{1st order dfe 3.1}\right) $, 
\begin{eqnarray}
\overset{\cdot }{x_{j}}+t_{s+1} &=&\tsum\nolimits_{i=1}^{s}t_{i}+t_{s+1}, 
\notag \\
t_{s+1} &=&\tsum\nolimits_{i=1}^{s+1}t_{i}-\overset{\cdot }{x_{j}}.
\end{eqnarray}%
Squaring equation $\left( \ref{1st order dfe 3.1.1}\right) ,$ yields%
\begin{eqnarray}
t_{s+1}^{2} &=&\left( \tsum\nolimits_{i=1}^{s+1}t_{i}-\overset{\cdot }{x_{j}}%
\right) \cdot \left( \tsum\nolimits_{\delta =1}^{s+1}t_{\delta }-\overset{%
\cdot }{x_{j}}\right)  \notag \\
t_{s+1}^{2} &=&\tsum\nolimits_{i=1}^{s+1}\tsum\nolimits_{\delta
=1}^{s+1}t_{i}t_{\delta }-\tsum\nolimits_{i=1}^{s+1}t_{i}\overset{\cdot }{%
x_{j}}-\overset{\cdot }{x_{j}}\tsum\nolimits_{\delta =1}^{s+1}t_{\delta }+%
\overset{\cdot }{x_{j}}\overset{\cdot }{x_{j}}  \notag \\
t_{s+1}^{2}
&=&t_{s+1}\tsum\nolimits_{i=1}^{s}t_{i}+\tsum\nolimits_{i=1}^{s}t_{i}\overset%
{\cdot }{x}_{j}-2\tsum\nolimits_{i=1}^{s+1}t_{i}\overset{\cdot }{x_{j}}%
+\tsum\nolimits_{i=1}^{s}t_{i}\overset{\cdot }{x_{j}}  \notag \\
t_{s+1}^{2}
&=&t_{s+1}\tsum\nolimits_{i=1}^{s}t_{i}-2t_{_{s+1}}\tsum%
\nolimits_{i=1}^{s}t_{i}=t_{_{s+1}}\tsum\nolimits_{i=1}^{s}t_{i}=0.
\end{eqnarray}%
Thus equation $\left( \ref{l+2 inner product 1}\right) $ yields an invariant
quantity%
\begin{eqnarray}
p_{A,l+2-s}\left( \psi \right) \circledast p_{A,l+2-s}\left( \psi \right) &=&%
\underset{R^{2}}{\underbrace{-\tsum\nolimits_{i=1}^{s}t_{i}^{2}+\tsum%
\nolimits_{j-1=s+1}^{k+1}x_{j}^{2}}}-t_{s+1}^{2}  \notag \\
p_{A,l+2-s}\left( \psi \right) \circledast p_{A,l+2-s}\left( \psi \right)
&=&R^{2}.
\end{eqnarray}%
Taking the derivative of equation $\left( \ref{x}\right) $ and using
equations $\left( \ref{1st order dfe 3.1}\right) $ and $\left( \ref{1st
order dfe 3.1.2}\right) ,$ we have%
\begin{eqnarray}
p\circledast \overset{\cdot }{p} &=&-\tsum\nolimits_{i=1}^{s+1}t_{i}\overset{%
\cdot }{t}_{i}+\tsum\nolimits_{j=s+2}^{k+2}x_{j}\overset{\cdot }{x}_{j} 
\notag \\
&=&-\tsum\nolimits_{i=1}^{s+1}t_{i}\tsum\nolimits_{j=s+2}^{k+2}x_{j}+\tsum%
\nolimits_{j=s+2}^{k+2}x_{j}\tsum\nolimits_{i=1}^{s+1}t_{i}  \notag \\
&=&-\tsum\nolimits_{i=1}^{s}t_{i}\tsum\nolimits_{j=s+2}^{k+2}x_{j}-t_{s+1}%
\tsum\nolimits_{j=s+2}^{k+2}x_{j}  \notag \\
&&+\tsum\nolimits_{j=s+2}^{k+2}x_{j}\tsum\nolimits_{i=1}^{s}t_{i}+\tsum%
\nolimits_{j=s+2}^{k+2}x_{j}t_{s+1}  \notag \\
&=&-\tsum\nolimits_{i=1}^{s}t_{i}\tsum\nolimits_{j=s+2}^{k+2}x_{j}+\tsum%
\nolimits_{j=s+2}^{k+2}x_{j}\tsum\nolimits_{i=1}^{s}t_{i}  \notag \\
&&-t_{s+1}\tsum\nolimits_{j=s+2}^{k+2}x_{j}+\tsum%
\nolimits_{j=s+2}^{k+2}x_{j}t_{s+1}  \notag \\
&=&-\tsum\nolimits_{i=1}^{s}t_{i}\tsum\nolimits_{j^{\prime
}=s+1}^{k+1}x_{j}+\tsum\nolimits_{j^{\prime
}=s+1}^{k+1}x_{j}\tsum\nolimits_{i=1}^{s}t_{i}=0.
\end{eqnarray}%
Adding $t_{s+1}$ to both sides of equation $\left( \ref{1st order dfe 3.1}%
\right) $, we have%
\begin{eqnarray}
\overset{\cdot }{x_{j}}+t_{s+1} &=&\tsum\nolimits_{i=1}^{s}t_{i}+t_{s+1}, 
\notag \\
t_{s+1} &=&\tsum\nolimits_{i=1}^{s+1}t_{i}-\overset{\cdot }{x_{j}}.
\label{1st order dfe 3.1.1}
\end{eqnarray}%
Squaring equation $\left( \ref{1st order dfe 3.1.1}\right) ,$ yields%
\begin{eqnarray}
t_{s+1}^{2} &=&\left( \tsum\nolimits_{i=1}^{s+1}t_{i}-\overset{\cdot }{x_{j}}%
\right) \left( \tsum\nolimits_{\delta =1}^{s+1}t_{\delta }-\overset{\cdot }{%
x_{j}}\right)  \notag \\
t_{s+1}^{2} &=&\tsum\nolimits_{i=1}^{s+1}\tsum\nolimits_{\delta
=1}^{s+1}t_{i}t_{\delta }-\tsum\nolimits_{i=1}^{s+1}t_{i}\overset{\cdot }{%
x_{j}}-\overset{\cdot }{x_{j}}\tsum\nolimits_{\delta =1}^{s+1}t_{\delta }+%
\overset{\cdot }{x_{j}}\overset{\cdot }{x_{j}}  \notag \\
t_{s+1}^{2}
&=&t_{s+1}\tsum\nolimits_{i=1}^{s}t_{i}+\tsum\nolimits_{i=1}^{s}t_{i}\overset%
{\cdot }{x}_{j}-2\tsum\nolimits_{i=1}^{s+1}t_{i}\overset{\cdot }{x_{j}}%
+\tsum\nolimits_{i=1}^{s}t_{i}\overset{\cdot }{x_{j}}  \notag \\
t_{s+1}^{2}
&=&t_{s+1}\tsum\nolimits_{i=1}^{s}t_{i}-2t_{_{s+1}}\tsum%
\nolimits_{i=1}^{s}t_{i}=t_{_{s+1}}\tsum\nolimits_{i=1}^{s}t_{i}=0.
\end{eqnarray}%
Thus equation $\left( \ref{l+2 inner product 1}\right) $ yields an invariant
quantity%
\begin{eqnarray}
p_{A,l+2-s}\left( \psi \right) \circledast p_{A,l+2-s}\left( \psi \right) &=&%
\underset{R^{2}}{\underbrace{-\tsum\nolimits_{i=1}^{s}t_{i}^{2}+\tsum%
\nolimits_{j-1=s+1}^{k+1}x_{j}^{2}}}-t_{s+1}^{2}  \notag \\
p_{A,l+2-s}\left( \psi \right) \circledast p_{A,l+2-s}\left( \psi \right)
&=&R^{2}.
\end{eqnarray}%
Therefore true for $l=k+1.$\pagebreak
\end{proof}

\end{document}